\documentclass[leqno,a4paper]{article}
\usepackage[english]{babel}
\usepackage{amsmath}
\usepackage{amsthm}
\usepackage{amssymb}
\usepackage{enumerate}
\usepackage{xspace}
\usepackage{euscript}
\usepackage{graphicx}
\usepackage{graphics}
\usepackage{amscd}
\usepackage{epsfig}
\usepackage{tabularx}
        \newtheorem{lemma}{Lemma}[section]
        
        \newtheorem{theorem}[lemma]{Theorem}
        \newtheorem{definition}{Definition}[section]
        \newtheorem{remark}[lemma]{Remark}
        \newtheorem{ex}[lemma]{Example}
\newtheorem{claim}[lemma]{Claim}

\numberwithin{equation}{section}

\title{\bf{Uniqueness for the electrostatic inverse boundary value problem with piecewise constant anisotropic conductivities}}
\author{Giovanni Alessandrini\thanks{Dipartimento di Matematica e Geoscienze, Universit\`{a} di Trieste, Italy. Email:alessang@units.it.}\qquad Maarten V. de Hoop\thanks{Departments of Computational and Applied Mathematics, Earth Science, Rice University, Houston, Texas, USA. Email:mdehoop@rice.edu}\qquad\\ Romina Gaburro\thanks{Department of Mathematics and Statistics, University of Limerick, Ireland.  Email: romina.gaburro@ul.ie}}

\date{\today}

\begin{document}
\maketitle

\begin{abstract}
We discuss the inverse problem of determining the, possibly anisotropic, conductivity of a body  $\Omega\subset\mathbb{R}^{n}$ when the so-called Neumann-to-Dirichlet map is locally given on a non empty curved portion $\Sigma$ of the boundary $\partial\Omega$. We prove that anisotropic conductivities that are \textit{a-priori} known to be piecewise constant matrices on a given partition of $\Omega$ with curved interfaces can be uniquely determined in the interior from the knowledge of the local Neumann-to-Dirichlet map.
\end{abstract}

\section{Introduction}\label{sec1}
\setcounter{equation}{0}
The inverse problem of recovering the conductivity of a body by taking measurements of voltage and current on its surface is studied in the present paper. More specifically, the case when the conductivity is anisotropic and it is \textit{a-priori} known to be a piecewise-constant matrix on a given partition of a domain (the body under investigation) is considered.  It is well-known that in absence of internal sources, the electrostatic potential $u$ in a conducting body, described by a domain $\Omega\subset{\mathbb R}^n$, is governed by the elliptic equation

\begin{equation}\label{eq conduttivita'}
\mbox{div}(\sigma\nabla{u})=0\qquad\mbox{in}\quad\Omega,
\end{equation}

where the symmetric, positive definite matrix $\sigma=\sigma(x)$, $x\in\Omega$ represents the (possibly anisotropic) electric conductivity. The inverse conductivity problem consists of finding $\sigma$ when the so called Dirichlet-to-Neumann (D-N) map
$$
\Lambda_{\sigma}:u\vert_{\partial\Omega}\in{H}^{\frac{1}{2}}(\partial\Omega)
\longrightarrow{\sigma}\nabla{u}\cdot\nu\vert_{\partial\Omega}\in{H}^{-\frac{1}{2}}(\partial\Omega)
$$
is given for any $u\in{H}^{1}(\Omega)$ solution to (1.1). Here, $\nu$ denotes the unit outer normal to $\partial\Omega$. If measurements can be taken only on one portion $\Sigma$ of $\partial\Omega$, then the relevant map is called the local D-N map ($\Lambda_{\sigma}^{\Sigma}$).


Different materials display different electrical properties, so that a map
of the conductivity $\sigma(x)$, $x\in\Omega$ can be used to investigate internal properties of
$\Omega$. This problem has many important applications in fields such as
geophysics, medicine and non--destructive testing of materials.
The first mathematical formulation of the inverse conductivity problem is due to A. P. Calder\'{o}n \cite{C},  where he addressed the problem of whether it is possible to determine the (isotropic) conductivity $\sigma = \gamma I$ by the D-N map. \cite{C} opened the way to the solution to the uniqueness issue where one is asking whether $\sigma$ can be determined by the knowledge of $\Lambda_{\sigma}$
(or $\Lambda_{\sigma}^{\Sigma}$ in the case of local measurements). We introduce the following function spaces

\begin{equation*}
_{0}H^{\frac{1}{2}}(\partial \Omega)=\left\{f\in
H^{\frac{1}{2}}(\partial \Omega)\vert\:\int_{\partial\Omega}f\:
=0\right\},
\end{equation*}
\begin{equation*}
_{0}H^{-\frac{1}{2}}(\partial \Omega)=\left\{\psi\in
H^{-\frac{1}{2}}(\partial \Omega)\vert\:\langle\psi,\:1\rangle=0
\right\}.
\end{equation*}

Observe that the D-N map $\Lambda_{\sigma}$ maps onto $_{0}H^{-\frac{1}{2}}(\partial \Omega)$, and, when restricted to $_{0}H^{\frac{1}{2}}(\partial\Omega)$, it is injective with bounded inverse called the Neumann-to-Dirichlet (N-D) map. The precise definitions of the D-N, N-D and its local version will be given in section 2. For now, we simply recall that the N-D map associates to specified current densities supported
on a portion $\Sigma\subset\partial\Omega$ the corresponding boundary voltages, also measured on the same portion $\Sigma$ of $\partial\Omega$ and that, mainly for the applications of the inverse conductivity problem to the direct-current (DC) resistivity method that we have in mind, the choice of taking the surface measurements by means of the (local) N-D map over the (local) D-N map seems to be more appropriate.

The case when measurements can be taken all over the boundary has been studied extensively in the past and fundamental papers like \cite{Ko-V1}, \cite{Ko-V2},  \cite{Sy-U} , \cite{N} and \cite{A} show that the isotropic case can be considered solved. More recently these uniqueness results have been extended in dimension $n \geq 3$ for conductivities in $C^1$ \cite{Ha-T}, for Lipschitz conductivities \cite{Ca-R} and for conductivities in $W^{s,p}(\Omega)\nsubseteq W^{1,\infty}(\Omega)$ \cite{Ha}, by assuming full boundary data. The original uniqueness result by Sylvester and Uhlmann \cite{Sy-U} required the conductivity to be $C^{\infty}$. For the two-dimensional case we refer to \cite{Bro-U} and the breakthrough paper \cite{As-P} where uniqueness has been proven for conductivities that are merely $L^{\infty}$. We wish to recall the uniqueness results of Druskin who, independently from Calder\'{o}n, dealt directly with the geophysical setting of the problem in \cite{D1}-\cite{D3} and that, in particular, the uniqueness result obtained in \cite{D2} was for conductivities described by piecewise constant functions (see also \cite{A-V}). In the present paper, we consider conductivities that are piecewise constant matrices. We refer to \cite{Bo}, \cite{C-I-N} and \cite{U} for an overview regarding the issues of uniqueness and reconstruction of the conductivity.

The problem of recovering the conductivity $\sigma$ by local measurements has been treated more recently. Lassas and Uhlmann in \cite{La-U} recovered a connected compact real-analytic Riemannian manifold $(M,\:g)$ with boundary by making use of the Green's function of the Laplace-Beltrami operator $\Delta_{g}$. See also \cite{La-U-T}. For the procedure of reconstructing the conductivity at the boundary by local measurements we refer to \cite{Bro}, \cite{NaT1}, \cite{NaT2}, \cite{K-Y}. An overview on reconstructing formulas of the conductivity and its normal derivative can be found in \cite{NaT3}. For related results of uniqueness in the interior in the case of local boundary data, we refer to Bukhgeim and Uhlmann \cite{B-U}, Kenig, Sj\"ostrand and  Uhlmann \cite{Ke-S-U} and Isakov \cite{Is}, and, for stability, Heck and Wang \cite{He-W}. Results of stability for cases of piecewise constant conductivities and local boundary maps have also been obtained in \cite{A-V}, \cite{Be-Fr} and \cite{D}.\\

On the other hand the anisotropic case is still open. Since Tartar's observation \cite{Ko-V1} that any diffeomorphism of $\Omega$ which keeps the boundary points fixed has the property of leaving the D-N map unchanged, whereas $\sigma$ is modified, different lines of research have been pursued. One direction has been to find the conductivity up to a diffeomorphism which keeps the boundary fixed (see \cite{L-U}, \cite{Sy}, \cite{N}, \cite{La-U}, \cite{La-U-T}, \cite{Be} and \cite{As-La-P}).  Another direction has been the one to formulate suitable \textit{a-priori} assumptions (possibly fitting some real life physical context) which constrain the structure of the unknown anisotropic conductivity. For instance, one can formulate the hypothesis that the directions of anisotropy are known while some scalar space dependent parameter is not, along this line of reasoning we mention \cite{Ko-V1}, \cite{A}, \cite{A-G}, \cite {A-G1}, \cite{G-L}, \cite{G-S} and \cite{L}. The case when $n=2$ and the anisotropic conductivity is assumed to be divergence free has been treated in \cite{A-Cab}.

Here we follow this second direction by \textit{a-priori} assuming that the conductivity is piecewise constant in a known finite partition of the domain, whereas the constant, matrix-valued, conductivities in each subdomain are unknown. An additional (apparently necessary) assumption that we pose is that contiguous subdomains of the partition can be joined by \textit{curved} smooth surfaces and also that the boundary portion $\Sigma$ where measurements are collected also contains a \textit{curved} portion of a surface. Under such assumptions we show, Theorem \ref{teorema principale}, that a local boundary map uniquely determines the conductivity, also in the interior. For the sake of concreteness we focus our analysis on the local N-D map. But it will be evident from the proof that also other choices of the boundary maps could be treated.

Let us outline the underlying ideas in our approach. As is well-known, \cite{B-G-M}, \cite{U}, the solutions to equation \eqref{eq conduttivita'} are the harmonic functions on the Riemannian manifold $\left\{\Omega,g\right\}$ where the metric $g$ is linked to the conductivity $\sigma$ through the relation

\begin{equation*}
g = \left(\det\sigma\right)^{\frac{1}{n-2}}\sigma^{-1}.
\end{equation*}

We shall obtain, Lemma \ref{lemma tau tangential}, that, under few regularity assumptions, from the knowledge of the local N-D map near a point $P \in \partial \Omega$, one can uniquely determine the tangential part of $g(P)$, that is the $(n-1)\times(n-1)$ minor of $g(P)$ relative to the tangent (hyper)plane to $\partial \Omega$ at $P$. Next, if the local N-D map in known on a non-flat portion $\Sigma$ of $\partial \Omega$ and $\sigma$ is constant nearby, then we have enough different tangent planes to completely recover $g$, and hence $\sigma$, Lemma \ref{lemma sigma constant}. The proof is then completed by an iteration argument and by the use of the unique continuation property.

Finally, in Example \ref{example}, which is a variation of the celebrated Tartar's example, \cite{Ko-V1}, we show that the N-D map for the half space is \textit{not} sufficient to uniquely determine a \textit{constant} anisotropic conductivity. Thus, this example provides a strong indication that indeed, flat boundary and interfaces may constitute an obstruction to uniqueness and thus our assumptions on curved interfaces and boundary are well-motivated. For other kinds of examples of nonuniqueness we may refer to \cite{G-L-U1} and \cite{G-L-U2}. \\

As early as 1920, Conrad Schlumberger \cite{Sc} recognized that anisotropy may affect geological formations' DC electrical properties. Anisotropic effects when measuring electromagnetic fields in geophysical applications have been studied ever since. From an inverse problems perspective, it is interesting that Maillet and Doll \cite{M-Do} already identified obstructions to recovering an anisotropic resistivity from (boundary) data.

Individual minerals are typically anisotropic but rocks composed of them can appear to be isotropic. Simpson and Tommasi \cite{Si-T} discussed the application of effective medium models to calculate the (degree of) anisotropy in electrical conductivity in an aggregate with non-random crystallographic orientations. In fact, there are many heterogeneous material configurations in Earth's sedimentary basins that possibly lead to anisotropy \cite{Ne-S}. It might be that there are some preferred directions in the subsurface rocks, or some preferred orientation of grains in the sediments. Fine layering or a pronounced strike direction can lead to an effective anisotropy. For example, alternations of sandstone and shales can cause hydrocarbon reservoir anisotropy, but anisotropy in shale-free sandstones can occur as well \cite{Ken-H}. Resistivity anisotropy has also been measured in volcanic reservoir rock \cite{No}.

In porous rocks, one simple equation that gives a relationship between their resistivities and the containing fluid saturation factor is Archie's law \cite{Ar}. This law is applicable for certain types of rocks and sediments, particularly those that have a low clay content. On the one hand, resistive fluids (hydrocarbons) displacing conductive ones (water) increase resistivity anisotropy in shaly rocks with the shale taking over the electrical conduction. On the other hand, anisotropy in Archie's law (through its parameters, see, for example, \cite{S-P-Lo}) is significant because permeability anisotropy can follow from it. That is, a formation factor can be extracted from Archie's law that can be anisotropic implying anisotropy in permeability through the tortuosity. In this context, we mention the work of Worthington \cite{Wo}.

In view of practical constraints on the data acquisition, DC resistivity methods are limited to probing Earth's (upper) crust. Resolving conductive structures to depths of the upper mantle requires magnetotelluric (MT) data. The analysis of the MT inverse boundary value problem associated with the low-frequency Maxwell equations will be presented in a separate paper. Most minerals in Earth's deeper interior (lower crust, upper mantle and transition zone) have been shown to have anisotropic conductivities that are sensitive not only to temperature, but also to hydrogen (water) content, major element chemistry and oxygen fugacity \cite{Ka-W}. Consequently, there is a potential to infer the distribution of these chemical factors (as well as temperature) from the study of electrical conductivities. Here, the influence of partial melting~\footnote{Melts in general have higher electrical conductivity than minerals. This is essentially due to the high diffusion coefficients of charged species in melts \cite{Ho}. As a consequence, the presence of partial melt will contribute to relatively high electrical conductivity.}  needs to be accounted for. Indeed, to infer the water distribution in Earth's mantle, electrical conductivity plays a primary role \cite{Ka}~\footnote{Hydrogen (water) has an important influence on rheological properties \cite{Ka-J} and melting relationship (\cite{Ku-Syo-Ak}, \cite{I}) that control the dynamics and evolution of our planet.}.


Many of the studies of anisotropy in as much as the solutions of the boundary value problem, in dimension three, and their probing capabilities are concerned, have been restricted to electrical conductivities (or resistivities) that are piecewise constant while plane layers form the subdomains in a domain partition of a half space. That is, flat interfaces separate the subdomains. Yin and Weidelt \cite{Y-We} considered arbitrary anisotropy for the DC-resistivity method in layered media.\\

The paper is organized as follows. Our main assumptions and our main result (Theorem \ref{teorema principale}) are contained in section 2, whereas section 3 contains some preliminary results. The proof of Theorem \ref{teorema principale}, that is, the proof of the unique determination of the piecewise constant anisotropic conductivity from the knowledge of the local N-D map, is contained in section 4. It should also be emphasized that the consideration of the local N-D map, rather than the local D-N map, is motivated by the application of this inverse problem to the DC resistivity method in geophysical prospection that we have in mind.


\section{Main Result}\label{sec2}
\setcounter{equation}{0}
\subsection{Notation and definition}

In several places in this manuscript it will be useful to single out one coordinate
direction. To this purpose, the following notations for
points $x\in \mathbb{R}^n$ will be adopted. For $n\geq 3$,
a point $x\in \mathbb{R}^n$ will be denoted by
$x=(x',x_n)$, where $x'\in\mathbb{R}^{n-1}$ and $x_n\in\mathbb{R}$.
Moreover, given a point $x\in \mathbb{R}^n$,
we shall denote with $B_r(x), B_r'(x)$ the open balls in
$\mathbb{R}^{n}, \mathbb{R}^{n-1}$ respectively centred at $x$ with radius $r$
and by $Q_r(x)$ the cylinder $B_r'(x')\times(x_n-t,x_n+r)$. We shall denote
$\mathbb{R}^n_+=\{(x',x_n)\in \mathbb{R}^n| x_n>0 \}, \
\mathbb{R}^n_-=\{(x',x_n)\in \mathbb{R}^n| x_n<0 \},
B^+_r=B_r\cap\mathbb{R}^n_+, B^-_r=B_r\cap\mathbb{R}^n_-,
Q^+_r=Q_r\cap\mathbb{R}^n_+, Q^{-}_r=Q_r\cap\mathbb{R}^n_-$.

In the sequel, we shall make a repeated use of quantitative
notions of smoothness for the boundaries of various domains. We introduce the following notation and definitions.

\begin{definition}\label{def Lipschitz boundary}
Let $\Omega$ be a domain in $\mathbb R^n$. We say that a portion
$\Sigma$ of $\partial\Omega$ is of Lipschitz class with constants
$r_0,L$ if for any $P\in\partial\Sigma$ there exists a rigid
transformation of $\mathbb R^n$ under which we have $P\equiv0$ and
$$\Omega\cap Q_{r_0}=\{x\in Q_{r_0}\,:\,x_n>\varphi(x')\},$$
where $\varphi$ is a Lipschitz function on $B'_{r_0}$ satisfying
the following condition $\varphi(0)=0$ and
$\|\varphi\|_{C^{0,1}(B'_{r_0})}\leq Lr_0$.

It is understood that $\partial\Omega$ is of Lipschitz class with
constants $r_0,L$ as a special case of $\Sigma$, with
$\Sigma=\partial\Omega$.
\end{definition}

\begin{definition}\label{def C1 alpha boundary}
Let $\Omega$ be a domain in $\mathbb R^n$. Given $\alpha$,
$\alpha\in(0,1)$, we say that a portion $\Sigma$ of
$\partial\Omega$ is of class $C^{1,\alpha}$  if for any $P\in\Sigma$ there exists a rigid transformation of
$\mathbb R^n$ under which we have $P=0$ and
$$\Omega\cap Q_{r_0}=\{x\in Q_{r_0}\,:\,x_n>\varphi(x')\},$$
where $\varphi$ is a $C^{1,\alpha}$ function on $B'_{r_0}$
satisfying
\[\varphi(0)=|\nabla_{x'}\varphi(0)|=0.\]




\end{definition}

\begin{definition}\label{definition D-N}
Given $\Sigma$ as above, we shall say that such a portion of a surface is non-flat if, there exists $P\in\Sigma$ such that, considering the  reference system and the function $\varphi$ as above, we have that $\varphi$ is not identically zero near $P=0$.
\end{definition}
\subsubsection{The Dirichlet-to-Neumann map.}\label{D-to-N}
We start by rigorously defining the D-N map.
We denote by $Sym_n$ the class of $n\times n$ symmetric real valued matrices.
Let $\Omega$ be a domain in $\mathbb{R}^n$ with Lipschitz boundary
$\partial\Omega$ and assume that $\sigma\in
L^{\infty}(\Omega\:,Sym_{n})$ satisfies the ellipticity condition
\begin{eqnarray}\label{ellitticita'sigma}
\lambda^{-1}\vert\xi\vert^{2}\leq{\sigma}(x)\xi\cdot\xi\leq\lambda\vert\xi\vert^{2},
& &for\:almost\:every\:x\in\Omega,\nonumber\\
& &for\:every\:\xi\in\mathbb{R}^{n}.
\end{eqnarray}

We shall also denote by $\langle\cdot,\cdot\rangle$ the
$L^{2}(\partial\Omega)$-pairing between
$H^{\frac{1}{2}}(\partial\Omega)$ and its dual
$H^{-\frac{1}{2}}(\partial\Omega)$.

\begin{definition}\label{def DN}
The Dirichlet-to-Neumann (D-N) map associated with $\sigma$ is the operator
\begin{equation}\label{mappaDN}
\Lambda_{\sigma}:H^{\frac{1}{2}}(\partial\Omega)\longrightarrow
{H}^{-\frac{1}{2}}(\partial\Omega)
\end{equation}

defined by

\begin{equation}\label{def DN locale}
\langle\Lambda_{\sigma}\:f,\:g\rangle\:=\:\int_{\:\Omega}
\sigma(x) \nabla{u}(x)\cdot\nabla\varphi(x)\:dx,
\end{equation}

for any $f$, $g\in H^{\frac{1}{2}}(\partial\Omega)$, where
$u\in{H}^{1}(\Omega)$ is the weak solution to

\begin{displaymath}
\left\{ \begin{array}{ll} \textnormal{div}(\sigma(x)\nabla
u(x))=0, &
\textrm{$\textnormal{in}\quad\Omega$},\\
u=f, & \textrm{$\textnormal{on}\quad{\partial\Omega},$}
\end{array} \right.
\end{displaymath}

and $\varphi\in H^{1}(\Omega)$ is any function such that
$\varphi\vert_{\partial\Omega}=g$ in the trace sense.
\end{definition}

Note that, by \eqref{def DN locale}, it is easily verified that
$\Lambda_{\sigma}$ is selfadjoint. Given $\sigma^{(i)}\in L^{\infty}(\Omega\:,Sym_{n})$, satisfying \eqref{ellitticita'sigma}, for $i=1,2$, we recall Alessandrini's identity (see \cite[(b), p. 253]{A})

\begin{equation}\label{Alessandrini identity}
\langle\left(\Lambda_{\sigma^{(1)}} - \Lambda_{\sigma^{(2)}}\right)f_1,f_2\rangle = \int_{\Omega} \left(\sigma^{(1)}(x) - \sigma^{(2)}(x)\right)\nabla u_1(x)\cdot\nabla u_2(x),
\end{equation}

for any $f_i\in H^{\frac{1}{2}}(\partial\Omega)$,  $i=1,2$ and $u_i\in H^{1}(\Omega)$ being the unique weak solution to the Dirichlet problem

\begin{displaymath}
\left\{ \begin{array}{ll} \textnormal{div}(\sigma^{(i)}(x)\nabla
u_i(x))=0, &
\textrm{$\textnormal{in}\quad\Omega$},\\
u_i=f_i, & \textrm{$\textnormal{on}\quad{\partial\Omega}$}.
\end{array} \right.
\end{displaymath}

We rigorously define now the local N-D map.

\subsubsection{The Neumann-to-Dirichlet map.}\label{N-to-D}
We consider the following function spaces

\begin{equation*}
_{0}H^{\frac{1}{2}}(\partial \Omega)=\left\{f\in
H^{\frac{1}{2}}(\partial \Omega)\vert\:\int_{\partial\Omega}f\:
=0\right\},
\end{equation*}
\begin{equation*}
_{0}H^{-\frac{1}{2}}(\partial \Omega)=\left\{\psi\in
H^{-\frac{1}{2}}(\partial \Omega)\vert\:\langle\psi,\:1\rangle=0
\right\}.
\end{equation*}

As previously observed, the D-N map
$\Lambda_{\sigma}$ maps onto $_{0}H^{-\frac{1}{2}}(\partial \Omega)$, and, when restricted to $_{0}H^{\frac{1}{2}}(\partial
\Omega)$, it is injective with bounded inverse. Then we can define the global Neumann-to-Dirichlet map as follows.

\begin{definition}\label{definition N-D}
The Neumann-to-Dirichlet (N-D) map associated with $\sigma$,

\[\mathcal{N}_{\sigma}:\  _{0}H^{-\frac{1}{2}}(\partial \Omega)\longrightarrow \:_{0}H^{\frac{1}{2}}(\partial \Omega)\]

is given by

\begin{equation}\label{Nsigma}
\mathcal{N}_{\sigma} = \left(\Lambda_{\sigma}|_{_{0}H^{\frac{1}{2}}(\partial
\Omega)} \right)^{-1} .\end{equation}
\end{definition}

Note that $\mathcal{N}_{\sigma}$ can also be characterized as the
selfadjoint operator satisfying
\begin{equation}\label{ND globale}
\langle\psi,\:\mathcal{N}_{\sigma}\psi\rangle\:=\:\int_{\:\Omega} \sigma(x)
\nabla{u}(x)\cdot\nabla{u}(x)\:dx,
\end{equation}
for every $\psi\in\: _{0}H^{-\frac{1}{2}}(\partial \Omega)$, where
$u\in{H}^{1}(\Omega)$ is the weak solution to the Neumann problem

\begin{equation}\label{N bvp}
\left\{ \begin{array}{lll}\displaystyle\textnormal{div}(\sigma\nabla u)=0, &
\textrm{$\textnormal{in}\quad\Omega$},\\
\displaystyle\sigma\nabla u\cdot\nu\vert_{\partial\Omega}=\psi, &
\textrm{$\textnormal{on}\quad{\partial\Omega}$},\\
\displaystyle\int_{\partial\Omega}u\: =0.
\end{array} \right.
\end{equation}

Given $\sigma^{(i)}\in L^{\infty}(\Omega\:,Sym_{n})$, satisfying \eqref{ellitticita'sigma}, for $i=1,2$, the following identity can be recovered from \eqref{Alessandrini identity}

\begin{equation}\label{Alessandrini identity N-D}
\langle\sigma^{(1)}\nabla u_1\cdot\nu,\left(\mathcal{N}_{\sigma^{(2)}} - \mathcal{N}_{\sigma^{(1)}}\right)\sigma^{(2)}\nabla u_2\cdot\nu\rangle = \int_{\Omega} \left(\sigma^{(1)}(x) - \sigma^{(2)}(x)\right)\nabla u_1(x)\cdot\nabla u_2(x),
\end{equation}

for any $u_i\in H^{1}(\Omega)$ weak solution to

\begin{equation}\label{conductivity equations}
\textnormal{div}(\sigma^{(i)}(x)\nabla u_i(x))=0,\qquad\textnormal{in}\quad\Omega,
\end{equation}

for $i=1,2$.\\

Now we introduce the local version of the N-D map. Let  $\Sigma$ be an open portion of $\partial\Omega$ and let
$\Delta=\partial\Omega\setminus\overline\Sigma$. We introduce the subspace of $H^{\frac{1}{2}}(\partial \Omega)$,

\[H^{\frac{1}{2}}_{co}(\Delta)=\left\{f\in H^{\frac{1}{2}}(\partial \Omega)\:|\: \mbox{supp}(f)\subset\Delta\right\}.\]

We denote by $H^{\frac{1}{2}}_{00}(\Delta)$  the closure in $H^{\frac{1}{2}}(\partial\Omega)$ of the space
$H^{\frac{1}{2}}_{co}(\Delta)$ and we introduce

\begin{equation}
_{0}H^{-\frac{1}{2}}(\Sigma)=\left\{\psi\in \:
_{0}H^{-\frac{1}{2}}(\partial\Omega)\vert\:\langle\psi,\:f\rangle=0,\quad\textnormal{for\:any}\:f\in
H^{\frac{1}{2}}_{00}(\Delta)\right\},
\end{equation}
that is the space of distributions $\psi \in
H^{-\frac{1}{2}}(\partial\Omega)$ which are supported in
$\overline\Sigma$ and have zero average on $\partial\Omega$. The local
N-D map is then defined as follows.

\begin{definition}
The local Neumann-to-Dirichlet map associated with $\sigma$,
$\Sigma$ is the operator $\mathcal{N}_{\sigma}^{\Sigma}:\:
_{0}H^{-\frac{1}{2}}(\Sigma)\longrightarrow
\big(_{0}H^{-\frac{1}{2}}(\Sigma)\big)^{\ast}\subset{_{0}H}^{\frac{1}{2}}(\partial\Omega)$
given  by
\begin{equation}
\langle \mathcal{N}_{\sigma}^{\Sigma}\;\varphi,\;\psi\rangle=\langle
\mathcal{N}_{\sigma}\;\varphi,\;\psi\rangle,
\end{equation}

for every $\varphi, \psi\in\: _{0}H^{-\frac{1}{2}}(\Sigma)$.
\end{definition}

Given $\sigma^{(i)}\in L^{\infty}(\Omega\:,Sym_{n})$, satisfying \eqref{ellitticita'sigma}, for $i=1,2$, we also recover from \eqref{Alessandrini identity}

\begin{equation}\label{Alessandrini identity local N-D}
\left<\psi_1,\left(\mathcal{N}_{\sigma^{(2)}}^{\Sigma} - \mathcal{N}_{\sigma^{(1)}}^{\Sigma}\right)\psi_2\right> = \int_{\Omega} \left(\sigma^{(1)}(x) - \sigma^{(2)}(x)\right)\nabla u_1(x)\cdot\nabla u_2(x),
\end{equation}

for any $\psi_i\in _{0}H^{-\frac{1}{2}}(\Sigma)$, for $i=1,2$ and $u_i\in H^{1}(\Omega)$ being the unique weak solution to the Neumann problem

\begin{equation}
\left\{ \begin{array}{lll}\displaystyle\textnormal{div}(\sigma^{(i)}\nabla u_i)=0, &
\textrm{$\textnormal{in}\quad\Omega$},\\
\displaystyle\sigma^{(i)}\nabla u_i\cdot\nu\vert_{\partial\Omega}=\psi_i, &
\textrm{$\textnormal{on}\quad{\partial\Omega}$},\\
\displaystyle\int_{\partial\Omega}u_i\: =0.
\end{array} \right.
\end{equation}



\subsection{The a-priori assumptions}\label{subsection assumptions}
Let $N$, $r_0$, $L$, $M$, $\alpha$, $\lambda$ be given
positive numbers with $N\in\mathbb{N}$, $\alpha\in(0,1)$. We will
refer to this set of numbers, along with the space dimension $n$,
as to the \textit{a-priori data}. For sake of simplicity we only consider $n\geq 3$.

\subsubsection{Assumptions pertaining to the domain partition}\label{subsec assumption domain}

\begin{enumerate}

\item $\Omega\subset\mathbb{R}^n$ is a bounded domain .

\item $\partial\Omega$ is of Lipschitz class.

\item We fix an open non-empty subset $\Sigma$ of $\partial\Omega$ (where the measurements in terms of the local N-D map are taken).

\item \[\overline\Omega = \bigcup_{j=1}^{N}\overline{{D}_j},\]

where $D_j$, $j=1,\dots , N$ are known open sets of
$\mathbb{R}^n$, satisfying the conditions below.

\begin{enumerate}
\item $D_j$, $j=1,\dots , N$ are connected and pairwise
nonoverlapping.

\item $\partial{D}_j$, $j=1,\dots , N$ are of Lipschitz class.

\item There exists one region, say $D_1$, such that
$\partial{D}_1\cap\Sigma$ contains a non flat $C^{1,\alpha}$ portion
$\Sigma_1$.

\item For every $i\in\{2,\dots , N\}$ there exists $j_1,\dots ,
j_K\in\{1,\dots , N\}$ such that

\begin{equation}\label{catena dominii}
D_{j_1}=D_1,\qquad D_{j_K}=D_i
\end{equation}

and such that for every $i=1,\dots , K$

\[\left(\bigcup_{k=1}^{i} \overline{D_{j_k}}\right)^{\circ}\quad\textnormal{ and}\quad \Omega\setminus\left(\bigcup_{k=1}^{i} \overline{D}_{j_k}\right)\]

are Lipschitz domains.

In addition we assume that, for every $k=1,\dots , K$,
$\partial{D}_{j_k}\cap \partial{D}_{j_{k-1}}$ contains a non flat
$C^{1,\alpha}$ portion $\Sigma_k$ (for the time being we agree that
$D_{j_0}=\mathbb{R}^n\setminus\Omega$), such that

\[\Sigma_1\subset\Sigma,\]

\[\Sigma_k\subset\Omega,\quad\mbox{for\:every}\:k=2,\dots , K,\]

and, for every $k=1,\dots , K$, there exists $P_k\in\Sigma_k$ and
a rigid transformation of coordinates under which we have $P_k=0$
and

\begin{eqnarray}
\Sigma_k\cap{Q}_{r_{0}/3} &=&\{x\in
Q_{r_0/3}|x_n=\varphi_k(x')\}\nonumber\\
D_{j_k}\cap {Q}_{r_{0}/3} &=&\{x\in
Q_{r_0/3}|x_n>\varphi_k(x')\}\nonumber\\
D_{j_{k-1}}\cap {Q}_{r_{0}/3} &=&\{x\in
Q_{r_0/3}|x_n<\varphi_k(x')\},
\end{eqnarray}

where $\varphi_k$ is a non flat $C^{1,\alpha}$ function on $B'_{r_o/3}$
satisfying

\[\varphi_k(0)=|\nabla\varphi_k(0)|=0.\]


\end{enumerate}
\end{enumerate}

\subsubsection{Assumption pertaining to the conductivity}

We assume that the conductivity $\sigma$ is of type

\begin{equation}\label{a priori info su sigma}
\sigma(x)=\sum_{j=1}^{N}\sigma_{j}\chi_{D_j}(x),\qquad
x\in\Omega,
\end{equation}

where $\sigma_{j}\in Sym_n$ are positive definite constant matrices, satisfying the uniform ellipticity condition

\begin{equation}\label{unifellip}
{\lambda}^{-1}|\xi|^2\le \sigma_j\xi\cdot\xi\le \lambda|\xi|^2 , \qquad\mbox{for every}\ \xi\in\mathbb{R}^n,
\end{equation}

for $j=1,\dots ,N$, and $D_j$, $j=1,\dots ,N$ are the subdomains introduced in section \ref{subsec assumption
domain} .


Our main result is stated below.

\begin{theorem}\label{teorema principale}
Let $\Omega$, $D_j$, $j=1,\dots , N$ and $\Sigma$ be a domain, $N$ subdomains of $\Omega$ and a portion of $\partial\Omega$ as in section \ref{subsec assumption domain} respectively and let $\sigma^{(i)}$, $i=1,2$ be two conductivities of type

\begin{equation}\label{conduttivita anisotrope}
\sigma^{(i)}(x)=\sum_{j=1}^{N}\sigma_{j}^{(i)}\chi_{D_j}(x)\qquad
x\in\Omega,\:i=1,2,
\end{equation}

where $\sigma_{j}^{(i)}\in Sym_n$ are positive definite constant matrices, satisfying the uniform ellipticity condition \eqref{unifellip}, for $j=1,\dots , N$. If

\[\mathcal{N}^{\Sigma}_{\sigma^{(1)}}=\mathcal{N}^{\Sigma}_{\sigma^{(2)}},\]

then

\begin{equation}\label{global uniqueness1}
\sigma^{(1)}=\sigma^{(2)},\qquad\textnormal{in}\quad\Omega.
\end{equation}

\end{theorem}


\section{The Neumann kernel}

From now on we shall denote by $\sigma(x)=\left\{\sigma_{ij}(x)\right\}_{i,j=1,\dots , n}$, $x\in\Omega$ a symmetric, positive definite matrix valued function satisfying \eqref{unifellip} and denote by $L$ the operator

\begin{equation}\label{operator L}
L=\mbox{div}\left(\sigma\nabla\cdot\right).
\end{equation}

We shall also introduce the matrix

\begin{equation}\label{tau}
g = \left(\det\sigma\right)^{\frac{1}{n-2}}\sigma^{-1}.
\end{equation}

\begin{remark}If we endow the open set $\Omega$ with the Riemannian metric $g$, then
\[ \frac{1}{\sqrt {det g}} L = \Delta_g, \]
that is, up to the factor $\frac{1}{\sqrt{det g}}$, the operator $L$ can be viewed as the Laplace-Beltrami operator for the Riemannian manifold $\left\{{M,g}\right\}$, see for instance \cite{B-G-M}, \cite{U}.
We emphasize that, being $n>2$, the knowledge of $\sigma$ is equivalent to the knowledge of $g$.
\end{remark}

We digress for a while and consider the operator \eqref{operator L} on a half space with $\sigma$ constant.
We denote by

\[\mathbb{R}^n_{+}=\left\{x=(x',x_n)\in\mathbb{R}^n\:|\:x_n>0\right\},\]

and by

\[\Pi_n =\left\{x=(x',x_n)\in\mathbb{R}^n\:|\:x_n=0\right\}\]

the half space in $\mathbb{R}^n$  and the hyperplane in $\mathbb{R}^n$ of points with vanishing $n^{th}$ coordinate respectively. From now on we will denote by $\xi\cdot\rho$ the Euclidean scalar product of vectors $\xi,\rho\in\mathbb{R}^n$.

Note that when $\sigma$ is constant, the same is true for $g$. We shall denote by $g_{(n-1)}$ the $(n-1)\times(n-1)$ submatrix of $g$ obtained by removing the $n^{th}$ row and column from $g$.

\begin{lemma}\label{lemma N and tau}
Let $N_{\sigma}$ be the Neumann kernel for the operator \eqref{operator L}, with $\sigma\in Sym_{n}$, on the half space $\mathbb{R}^{n}_{+}$. For every $x\in \mathbb{R}^{n}_{+}$ and $y'\in\Pi_n$ we have

\begin{equation}\label{N sigma constant}
N_{\sigma}(x,y') = 2C_n\left(g(x-y')\cdot(x-y')\right)^{\frac{2-n}{2}},
\end{equation}

where $C_{n}=1/n(n-2)\omega_n$ , with $\omega_n$ denoting the volume of the unit ball in $\mathbb{R}^n$. In particular, if $N_{\sigma}(x',y')$ is known for every $x',y'\in\Pi_n$ then $g_{(n-1)}$ is uniquely determined.
\end{lemma}

\begin{proof}
We temporarily set $\sigma =I$, where $I$ is the $n\times n$ identity matrix and let


\[
T = \left(
\begin{array}{c|c}
  \raisebox{-15pt}{{\huge\mbox{{$I_{(n-1)}$}}}} & 0 \\[-5ex]
  & \vdots \\[-0.5ex]
  & 0 \\ \hline
   0 \cdots 0& \!\!\!-1
\end{array}
\right),
\]

where $I_{(n-1)}$ denotes the $(n-1)\times (n-1)$ identity matrix.\\

It is well-known that the Neumann kernel for the Laplacian on the half space $\mathbb{R}^n_{+}$ is given by

\begin{equation}\label{neumann kernel}
N_{I}(x,y)=\Gamma(x-y)+\Gamma(x-Ty),
\end{equation}

for every $x,y\in\overline{\mathbb{R}^{n}_{+}}$, $x\neq y$. Here

\begin{equation}\label{fundamental solution laplace}
\Gamma(x)=C_{n}|x|^{2-n}\qquad
\end{equation}

is the fundamental solution for the Laplacian in $\mathbb{R}^n$ and

\begin{equation}\label{Cn}
C_{n}=\frac{1}{n(n-2)\omega_n},
\end{equation}

where $\omega_n$ denotes the volume of the unit ball in $\mathbb{R}^n$.  Let $M$ be an $n\times n$ invertible matrix such that $M\mathbb{R}^n_{+}=\mathbb{R}^n_{+}$. Let $Q=M^{-1}$ and consider the linear change of coordinates

\[\xi=Mx,\qquad x=Q\xi,\qquad\textnormal{for\:any}\quad x\in\mathbb{R}^n_{+}.\]

For every $\psi\in C^{0,1}_{0}(\mathbb{R}^n)$ and every $y\in\overline{\mathbb{R}^n_{+}}$ we have

\begin{equation}\label{I1}
\int_{\mathbb{R}^{n}_{+}} \nabla_{\xi}N_{I}(\xi,\eta)\cdot\nabla_{\xi}\psi(\xi)\:d\xi =\psi(\eta),
\end{equation}

where $\eta=My$. Changing variables

\begin{equation}\label{I2}
\int_{\mathbb{R}^{n}_{+}} \frac{QQ^{T}}{\det{Q}}\nabla_{x}N_{I}(Mx,My)\cdot\nabla_{x}\psi(Mx)\:dx =\psi(My).
\end{equation}

We fix an arbitrary matrix $\sigma\in Sym_n$, positive definite and constant. We look for a matrix $M=Q^{-1}$ as above such that

\begin{equation}\label{sigma, Q}
\frac{QQ^{T}}{\det(Q)}=\sigma.
\end{equation}

For this purpose, we set

\begin{equation}\label{Q}
Q=\alpha\sqrt{\sigma}R,
\end{equation}

where $\alpha>0$ is to be chosen later on, $\sqrt{\sigma}$ denotes the symmetric, positive definite, square root of $\sigma$ and $R$ is an orthogonal transformation chosen in such a way that

\[Q\mathbb{R}^n_{+}=\mathbb{R}^n_{+}\qquad(\textnormal{i.e.}\quad M\mathbb{R}^n_{+}=\mathbb{R}^n_{+}).\]

$R$ can be readily determined by prescribing

\[R\mathbb{R}^n_{+}=\sqrt{\sigma^{-1}}\mathbb{R}^n_{+}.\]

Next, we determine $\alpha$. Note that by \eqref{sigma, Q} we must have

\[\left(\det(Q)\right)^{2-n}=\det(\sigma)\]

whereas, by \eqref{Q}

\[\det(Q)=\alpha^n\left(\det(\sigma)\right)^{\frac{1}{2}},\]

hence

\[\alpha^{n}=\left(\det(\sigma)\right)^{\frac{n}{2(2-n)}},\]

that is

\[\alpha = \left(\det(\sigma)\right)^{\frac{1}{2(2-n)}}.\]

With the above choices, we obtain

\begin{equation}\label{N sigma constant 1}
N_{\sigma}(x,y)=N_{I}(Mx,My),
\end{equation}

or as is the same

\begin{equation}\label{N sigma constant 2}
N_{\sigma}(x,y)=C_n\:\left(|M(x-y)|^{2-n} + |Mx - TMy|^{2-n}\right),
\end{equation}

where $C_n$ is given by \eqref{Cn}. Let $S$ be such that $TM=MS$ that is

\[S=QTQ^{-1}=\sqrt{\sigma}RTR^{T}\sqrt{\sigma^{-1}}.\]

Note that if $y'\in\Pi_n$, then $Q^{-1}y'\in\Pi_n$, hence $TQ^{-1}y'=Q^{-1}y'$, therefore

\begin{equation}\label{S on the hyperplane}
Sy' = y',\qquad\textnormal{for\:every}\quad y'\in\Pi_n.
\end{equation}

Consequently

\begin{equation}\label{N sigma constant 3}
N_{\sigma}(x,y)=C_n\left(\left(M^{T}M(x-y)\cdot(x-y)\right)^{\frac{2-n}{2}} +
\left(M^{T}M(x-Sy)\cdot(x-Sy)\right)^{\frac{2-n}{2}} \right).
\end{equation}

We observe that

\[M^{T}M=\alpha^{-2}\sqrt{\sigma^{-1}}RR^{T}\sqrt{\sigma^{-1}}=\alpha^{-2}\sigma^{-1}=\left(\det(\sigma)\right)^{\frac{1}{n-2}}\sigma^{-1}=g.\]

Therefore

\begin{equation}\label{N sigma constant 4}
N_{\sigma}(x,y)=C_n\:\left(\left(g(x-y)\cdot(x-y)\right)^{\frac{2-n}{2}} +
\left(g(x-Sy)\cdot(x-Sy)\right)^{\frac{2-n}{2}} \right),
\end{equation}

hence \eqref{N sigma constant} follows and, in particular, when $x',y'\in\Pi_n$

\begin{equation}\label{N sigma constant 4 Pi}
N_{\sigma}(x',y')=2C_n\:\left(g(x'-y')\cdot(x'-y')\right)^{\frac{2-n}{2}}
\end{equation}

or as is the same

\[g(x'-y')\cdot(x'-y') = \left(\frac{N_{\sigma}(x',y')}{2C_{n}}\right)^{\frac{2}{2-n}},\quad\textnormal{for\:all}\quad x',y'\in\Pi_n.\]

Consequently $g_{(n-1)}$ is uniquely determined by $N_{\sigma}(x',y')$, $x',y'\in\Pi_n$.

\end{proof}

We shall also introduce the Neumann kernel $N^{\Omega}_{\sigma}$ for the boundary value problem associated with the operator \eqref{operator L} and $\Omega$ by defining it, for any $y\in\Omega$, $N^{\Omega}_{\sigma}(\cdot,y)$ to be the distributional solution to

\begin{displaymath}\label{def neumann kernel}
\left\{ \begin{array}{ll}
L\:N^{\Omega}_{\sigma}(\cdot,y)=-\delta(\cdot -y), & \textnormal{in}\quad\Omega\\
\sigma\nabla N^{\Omega}_{\sigma}(\cdot, y)\cdot\nu= -\frac{1}{\vert\partial\Omega\vert},
& \textnormal{on}\quad{\partial\Omega}.
\end{array} \right.
\end{displaymath}

Note that $N^{\Omega}_{\sigma}$ is uniquely determined up to an additive constant. For simplicity we impose the normalization

\[\int_{\partial\Omega} N^{\Omega}_{\sigma}(\cdot,y)\:dS(\cdot)=0.\]

With this convention we obtain by Green's identities that

\begin{equation}\label{symmetry of N}
N^{\Omega}_{\sigma}(x,y) = N^{\Omega}_{\sigma}(y,x),\qquad\textnormal{for\:all}\quad x,y\in\Omega,\quad x\neq y.
\end{equation}

\begin{remark}
$N^{\Omega}_{\sigma}(x,y)$ extends continuously up to the boundary $\partial\Omega$ (provided that $x\neq y$) and in particular, when $y\in\partial\Omega$, it solves

\begin{displaymath}\label{neumann kernel boundary}
\left\{ \begin{array}{ll}
L\:N^{\Omega}_{\sigma}(\cdot,y)=0, & \textnormal{in}\quad\Omega\\
\sigma\nabla N^{\Omega}_{\sigma}(\cdot, y)\cdot\nu= \delta(\cdot -y)-\frac{1}{\vert\partial\Omega\vert},
& \textnormal{on}\quad{\partial\Omega}.
\end{array} \right.
\end{displaymath}
\end{remark}

\begin{theorem}\label{theorem neumann function holder}
Let $y\in\partial\Omega$ and assume that there exists a neighbourhood $\mathcal{U}$ of $y$ such that $\partial\Omega\cap\mathcal{U}$ is a portion of class $C^{1,\alpha}$, with $0<\alpha<1$, of $\partial\Omega$ and $\sigma$ in \eqref{operator L} is such that $\sigma\in C^{\alpha}(\mathcal{U}\cap\overline\Omega)$. Then the Neumann kernel $N^{\Omega}_{\sigma}(\cdot,y)$ satisfies

\begin{equation}\label{neumann kernel holder coefficients}
N^{\Omega}_{\sigma}(x,y)=
2 C_{n}\:\big(\det(\sigma(y))\big)^{-1/2}\Big(\sigma^{-1}(y)(x-y)\cdot(x-y)\Big)^{\frac{2-n}{2}}+O(|x-y|^{2-n+\alpha}),
\end{equation}

\noindent as $x\rightarrow y$, $x\in\overline{\Omega}\setminus{\{y\}}$ and $C_{n}$ is the constant given in \eqref{Cn}.





\end{theorem}
\begin{proof}
This result has a classical flavour and is possibly well-known. We refer to \cite[Chapter 1]{Mi} and \cite[(1.31)-(1.33)]{Mit-T} for the case $\sigma\in C^{\alpha}(\Omega)$, with $\partial\Omega$ of class $C^{1,\alpha}$. We sketch a proof for the sake of completeness. We represent $\Sigma =\partial\Omega\cap\mathcal{U}$ according to definition \ref{def C1 alpha boundary}, and assume without loss of generality that $y=0$. Let $r>0$ be such that $\overline{B_{r}(0)}\subset\mathcal{U}$. For any $\psi\in C^{0,1}_{0}(B_{r}(0))$ we have

\begin{equation}\label{psi I1}
\int_{\Omega\cap B_{r}(0)} \sigma(x)\nabla_{x}N^{\Omega}_{\sigma}(x,0)\cdot\nabla_{x}\psi(x)\:dx =\psi(0)-\frac{1}{|\partial\Omega|}\int_{\partial\Omega\cap B_{r}(0)} \psi(x)\:dS(x).
\end{equation}

We introduce the change of coordinates $z=z(x)$ ($x=x(z)$)

\begin{displaymath}\label{change of coordinates}
\left\{ \begin{array}{ll}
z'  = x' & \\
z_n  = x_n - \varphi(x') &.
\end{array} \right.
\end{displaymath}

We have

\begin{equation}\label{z}
z=x+O(|x'|^{1+\alpha})
\end{equation}

and also, setting $J=\frac{\partial z}{\partial x}$,

\begin{equation}\label{J}
J=I+O(|x'|^{\alpha}).
\end{equation}

Next, we define

\begin{eqnarray}
\widetilde{\sigma}(z) & = &\left(\frac{1}{\det(J)}\:J\sigma J^{T}\right)(x(z))\label{sigma tilde1}\\
\widetilde{N}(z) & = & N^{\Omega}_{\sigma}\left(x(z),0\right)\label{N tilde}.
\end{eqnarray}

We obtain

\begin{equation}\label{psi I2}
\int_{\{z_n>0\}} \widetilde{\sigma}(z)\nabla_{z}\widetilde{N}(z)\cdot\nabla_{z}\psi(x(z))\:dz =\psi(0)-\frac{1}{|\partial\Omega |}\int_{\Pi_n} \psi(z',0)\sqrt{1+|\nabla_{z'}\varphi|^2}\:dz'.
\end{equation}

We denote

\[q(z')=\frac{1}{|\partial\Omega|}\:\sqrt{1+|\nabla_{z'}\varphi|^2}.\]

Note that $q$ is bounded and that

\begin{equation}\label{sigma tilde2}
\widetilde{\sigma}(z)=\sigma(0)+O(|z'|^{\alpha}).
\end{equation}

Let $N_0$ denote the Neumann function for $\mathbb{R}^{n}_{+}$ with $\sigma = \sigma(0)$ and denote

\begin{equation}\label{R}
R(z) = \widetilde{N}(z) - N_{0}(z,0).
\end{equation}

We have

\begin{multline}\label{psi I3}
\int_{\{z_n>0\}} \widetilde{\sigma}(0)\nabla_{z}R(z)\cdot\nabla_{z}\psi(x(z))\:dz \\
= \int_{\{z_n>0\}}\left(\widetilde{\sigma}(0) - \widetilde{\sigma}(z)\right)\nabla_{z}\widetilde{N}(z)\cdot\nabla_{z}\psi(x(z))\:dz
 - \int_{\Pi_n} \psi(z',0)q(z')\:dz'.
 \end{multline}

Hence, for a sufficiently small $\rho>0$ we have


\begin{displaymath}\label{problem for R}
\left\{ \begin{array}{ll}
\textnormal{div}_z\left(\sigma(0)\nabla_z R(z)\right) = \textnormal{div}_z\left((\sigma(0) - \widetilde{\sigma}(z))\nabla_z \widetilde{N}(z)\right)  & \textnormal{in}\quad B_{\rho}(0)\cap\mathbb{R}^{n}_{+}\\
\sigma(0)\nabla_z R(z)\cdot\nu= (\sigma(0) - \widetilde{\sigma}(z))\nabla_z \widetilde{N}(z)\cdot\nu - q(z') & \textnormal{in}\quad B_{\rho}(0)\cap\Pi_n.
\end{array} \right.
\end{displaymath}

We recall that

\begin{equation}\label{estimate N}
|N^{\Omega}_{\sigma}(x,0)|\leq C|x|^{2-n},\qquad\textnormal{for\:every}\quad x\in\Omega,
\end{equation}

where $C>0$ is a constant that only depends on ellipticity and on the Lipschitz regularity of $\partial\Omega$ (see e.g. \cite{Ke-P}). Next using the local regularity of $\sigma$ and of $\Sigma\subset\partial\Omega$ we also obtain

\begin{equation}\label{estimate gradient N}
|\nabla_{x}N^{\Omega}_{\sigma}(x,0)|\leq C|x|^{1-n},\qquad\textnormal{for\:every}\quad x\in B_{\rho}(0)\cap\Omega.
\end{equation}

Consequently

\begin{equation}\label{estimate R and gradient of R}
|R(z)|+|z|\:|\nabla_{z} R(z)|\leq C\qquad\textnormal{for\:every}\quad z\in \partial{B}_{\rho}(0)\cap\mathbb{R}^{n}_{+}.
\end{equation}

By Green's identities, setting $B^{+}_{\rho}=B_{\rho}(0)\cap\mathbb{R}^{n}_{+}$, for every $w\in B^{+}_{\rho}$ we obtain

\begin{eqnarray}\label{R I1}
R(w) &=& -\int_{B^{+}_{\rho}} R(z)\:\textnormal{div}_z\left(\sigma(0)\nabla_z N_{0}(z,w)\right)\:dz\nonumber\\
&=& -\int_{\partial{B}^{+}_{\rho}} \Big(R(z)\sigma(0)\nabla_z N_{0}(z,w)\cdot\nu - N_{0}(z,w)\sigma(0)\nabla_z R(z)\cdot\nu\Big)\:dS(z)\nonumber\\
&&- \int_{B^{+}_{\rho}} N_{0}(z,w)\:\textnormal{div}_z\left(\sigma(0)\nabla_z R(z)\right)\:dz\nonumber\\
&=& -\int_{\partial{B}^{+}_{\rho}} \Big(R(z)\sigma(0)\nabla_z N_{0}(z,w)\cdot\nu - N_{0}(z,w)\sigma(0)\nabla_z R(z)\cdot\nu\Big)\:dS(z)\nonumber\\
&&- \int_{\partial{B}^{+}_{\rho}} N_{0}(z,w)\left(\sigma(0)-\widetilde{\sigma}(z)\right)\nabla_{z}\widetilde{N}(z,w)\cdot\nu\:\:dS(z)\nonumber\\
&&+ \int_{B^{+}_{\rho}} \left(\sigma(0)-\widetilde{\sigma}(z)\right)\nabla_{z}N_{0}(z,w)\cdot\nabla_{z}\widetilde{N}(z)\:\:dz.
\end{eqnarray}

If we split $\partial{B}^{+}_{\rho}=\left(\partial{B}_{\rho}\cap\mathbb{R}^{n}_{+}\right)\cup\left(B_{\rho}\cap\Pi_n\right)$, we obtain

\begin{eqnarray}\label{R I2}
R(w) &=& -\int_{\partial{B}_{\rho}\cap\mathbb{R}^{n}_{+}} \Big(R(z)\sigma(0)\nabla_z N_{0}(z,w)\cdot\nu - N_{0}(z,w)\sigma(0)\nabla_z R(z)\cdot\nu\Big)\:dS(z)\nonumber\\
&&- \int_{\partial{B}_{\rho}\cap\mathbb{R}^{n}_{+}} N_{0}(z,w)\left(\sigma(0)-\widetilde{\sigma}(z)\right)\nabla_{z}\widetilde{N}(z,w)\cdot\nu\:\:dS(z)\nonumber\\
& & - \int_{B_{\rho}\cap\Pi_n} N_{0}(z',w)q(z')\:dz' \nonumber\\
&&+ \int_{B^{+}_{\rho}} \left(\sigma(0)-\widetilde{\sigma}(z)\right)\nabla_{z}N_{0}(z,w)\cdot\nabla_{z}\widetilde{N}(z)\:\:dz.
\end{eqnarray}

Taking $|w|<\frac{\rho}{2}$ all the boundary integrals in \eqref{R I2} are uniformly bounded. Whereas the volume integral appearing in \eqref{R I2}, in view of \eqref{N sigma constant 4} and of \eqref{estimate gradient N}, can be estimated as follows

\begin{eqnarray}\label{estimate R I2 volume}
& &\left|\int_{B^{+}_{\rho}} \left(\sigma(0)-\widetilde{\sigma}(z)\right)\nabla_{z}N_{0}(z,w)\cdot\nabla_{z}\widetilde{N}(z)\:\:dz\right|\nonumber\\
& &\leq  C\int_{B^{+}_{\rho}} |z'|^{\alpha}\:|z-w|^{1-n}\:|z|^{1-n}\:dz\leq C|w|^{2-n+\alpha},
\end{eqnarray}

hence $|R(z)|\leq C |z|^{2-n+\alpha}$ on $B^{+}_{\rho}$ and recalling that $|z|=O(|x|)$ the thesis follows.

\end{proof}

Therefore we have

\begin{lemma}\label{lemma tau tangential}
If $y'\in\partial\Omega$ and there is a neighbourhood $\mathcal{U}$ of $y'$ such that $\partial\Omega\cap\mathcal{U}$ is a portion of $\partial\Omega$ of class $C^{1,\alpha}$ and $L$ is the operator \eqref{operator L}, with coefficients matrix $\sigma\in C^{\alpha}(\mathcal{U}\cap\overline\Omega)$, with $0<\alpha<1$, then the knowledge of $N^{\Omega}_{\sigma}(x,y')$, for every $x\in\partial\Omega\cap\mathcal{U}$ uniquely determines

\begin{equation}\label{tau boundary}
g_{(n-1)}(y')=\left\{g(y')v_i \cdot v_j\right\}_{i,j=1,\dots , (n-1)},
\end{equation}

where $v_1,\dots ,v_{n-1}$  is a basis for $T_{y'} (\partial\Omega)$, the tangent plane to $\partial\Omega$ at $y'$.
\end{lemma}

\begin{proof}
Without loss of generality we choose a coordinate system at $y'\in\partial\Omega$ such that $y'=0$ and the tangent plane to $\partial\Omega$ at $y'$ is $T_0 (\partial\Omega)=\Pi_n$. For any $\xi\in\Pi_n$, $|\xi|=1$, we choose $x'=r\xi$, with $r$ small and denote $x=(x',\varphi(x'))\in\partial\Omega$, then by \eqref{neumann kernel holder coefficients}

\[\lim_{r\rightarrow 0} N^{\Omega}_{\sigma}(x,y')\: r^{\frac{n-2}{2}}=2C_n\left(g(y')\xi \cdot \xi\right)^{\frac{2-n}{2}},\]

for all $\xi\in\Pi_n$, $|\xi|=1$. Hence $g_{(n-1)}(y')$ is uniquely determined.
\end{proof}

\begin{lemma}\label{lemma sigma constant}
Let $\Omega$ be a domain in $\mathbb{R}^n$ with boundary $\partial\Omega$ of Lipschitz class and let $\Sigma$ be an open portion of $\partial\Omega$ of class $C^{1,\alpha}$ and non flat near some point $y'_0\in\Sigma$. If $\sigma\in L^{\infty}(\Omega\:,Sym_{n})$ satisfies \eqref{ellitticita'sigma} and it is constant near $y'_0$ and $\Sigma$, then the knowledge of $N^{\Omega}_{\sigma}(x',y')$, for every $x',y'\in\Sigma$ uniquely determines $\sigma(y'_0)$.
\end{lemma}

\begin{proof}
We denote by $\{e_1,\dots , e_n\}$ the canonical basis in $\mathbb{R}^n$. We assume, without loss of generality, that $y'_0=0\in\Sigma$, that the tangent space to $\partial\Omega$ at $0\in\Sigma$ is $T_0 (\partial\Sigma)=\Pi_n=<e_1,\dots , e_{n-1}>$ and the outer unit normal to $\partial\Omega$ at $0$ is $-e_n$. For any $P\in\partial\Omega$, we will denote by $\nu(P)$ the outer unit normal to $\partial\Omega$ at $P$ ($\nu(0)=-e_n$). If $\Sigma$ is not flat near $0$, then there are points $P\in\Sigma$ nearby such that $\nu(P)$ slightly deflects from $\nu(0)=-e_n$, therefore without loss of generality we can assume that there exists a point $P\in\Sigma$ and some $\varepsilon\neq 0$ such that

\begin{equation}\label{normal near 0 case 1}
\nu(P)=\frac{1}{\sqrt{1+\varepsilon^{2}}}\left(-e_n+\varepsilon e_{n-1}\right).
\end{equation}

Depending on the geometry of $\Sigma$ near $0$, there is an alternative:

\begin{enumerate}[(a)]

\item The deflection of $\nu$ is everywhere in the $e_{n-1}$ direction.

\item There are points $\widetilde{P}\in\Sigma$ near $0$ in which the deflection of $\nu$ is in a direction independent of $e_{n-1}$ and without loss of generality we can assume that there is a point $\widetilde{P}\in\Sigma$ and some $\alpha,\beta\in\mathbb{R}$, with $\alpha\neq 0$ such that

\begin{equation}\label{normal near 0 case 2}
\nu(\widetilde{P})=\frac{1}{\sqrt{1+\alpha^{2}+\beta^{2}}}\left(-e_n+\alpha e_{n-2}+\beta e_{n-1}\right).
\end{equation}
\end{enumerate}

Next, we show that in either cases (a) and (b), $g(0)$ (hence $\sigma(0)$) can be uniquely determined. We denote by

\[g=g(0)\]

and start with case (a). In this case an orthonormal basis for the tangent space $T_P(\Sigma)$ is given by

\begin{equation}\label{basis tangent at P}
\left\{e_1,\dots e_{n-2}, \frac{1}{\sqrt{1+\varepsilon^2}}\:e_{n-1} + \frac{\varepsilon}{\sqrt{1+\varepsilon^2}}\:e_{n}\right\}.
\end{equation}

Suppose $\varepsilon>0$. By continuity, we can find a continuous path $Q=Q(t)$, for $0<t<\varepsilon$ along $\Sigma$ such that $Q(0)=0$, $Q(\varepsilon)=P$, $g(Q(t))=g$, $0<t<\varepsilon$ and such that an orthonormal basis for the tangent space $T_{Q(t)}(\Sigma)$ is given by

\begin{equation}\label{basis tangent at Q(t)}
\left\{e_1,\dots e_{n-2}, \frac{1}{\sqrt{1+t^2}}\:e_{n-1} + \frac{t}{\sqrt{1+t^2}}\:e_{n}\right\}.
\end{equation}

Recalling that by Lemma \ref{lemma tau tangential} we know

\begin{equation}\label{tau tangential}
g v_i\cdot v_j ,\qquad i,j=1,\dots, n-1,
\end{equation}

for all $v_i$, $i=1,\dots , n-1$, forming a basis for $T_{Q(t)}\Sigma$, for any $t$, $0<t<\varepsilon$, we have that  the following functions

\begin{eqnarray}
& &g e_i\cdot\left( \frac{1}{\sqrt{1+t^2}}\:e_{n-1} + \frac{t}{\sqrt{1+t^2}}\:e_{n}\right),\label{info tau 1}\\
& &g\left(\frac{1}{\sqrt{1+t^2}}\:e_{n-1} + \frac{t}{\sqrt{1+t^2}}\:e_{n}\right)\cdot\left(\frac{1}{\sqrt{1+t^2}}\:e_{n-1} + \frac{t}{\sqrt{1+t^2}}\:e_{n}\right)\label{info tau 2}
\end{eqnarray}

are known for any $i=1,\dots , n-1$ and any $t$, $0<t<\varepsilon$. From \eqref{info tau 1} we obtain that the function

\begin{equation}
g_{i,\:n-1}+tg_{i,\:n}
\end{equation}

is known for any any $t$, $0<t<\varepsilon$, for any $i=1,\dots , n-2$ and hence $g_{i,\:n}$ is known for any $i=1,\dots , n-2$. From \eqref{info tau 2} we obtain that the polynomial

\begin{equation}
g_{n-1,\:n-1}+2tg_{n-1,\:n}+t^2g_{n,n}
\end{equation}

is known for any $t$, $0<t<\varepsilon$, hence all of its coefficients are known , in particular $g_{n-1,\:n}$ and $g_{n,n}$ are known too, therefore the full matrix $g$ is determined in case (a).

Next, we consider case (b). For $\widetilde{P}$ near $0$, we have that

\[g(\widetilde{P})=g\]

and that $g_{i,j}$ is known for any $i,j=1,\dots , n-1$ by Lemma \ref{lemma tau tangential}. $g_{i,n}$ is also known for $i=1,\dots , n-2$ by recalling that the following scalar product

\begin{equation*}
g e_i\cdot\left(\frac{1}{\sqrt{1+\varepsilon^2}}\:e_{n-1} + \frac{t}{\sqrt{1+\varepsilon^2}}\:e_{n}\right)
\end{equation*}

is known. To determine the remaining entries $g_{n-1,n}$ , $g_{n,n}$ of the matrix $g$, we note that a basis for the tangent space $T_{\widetilde{P}}\Sigma$ is given by

\begin{equation}\label{basis tangent at Q}
\left\{e_1,\dots e_{n-3}, e_{n-2}+\alpha e_n, e_{n-1}+\beta e_n\right\}.
\end{equation}

The following expressions

\begin{eqnarray}
& & g\left(e_{n-2}+\alpha e_n\right)\cdot\left(e_{n-2}+\alpha e_n\right),\label{info v v}\\
& & g\left(e_{n-1}+\beta e_n\right)\cdot\left(e_{n-2}+\alpha e_n\right)\label{info w v}
\end{eqnarray}

are known and from \eqref{info v v}, \eqref{info w v} we recover that the following expressions

\begin{eqnarray}
& &g_{n-2,n-2}+2\alpha g_{n-2,n}+\alpha^{2}g_{n,n},\label{tau n,n}\\
& &g_{n-1,n-2}+\beta g_{n,n-2}+\alpha g_{n-1,n}+\alpha\beta g_{n,n}\label{tau n-1,n}
\end{eqnarray}

are known too. From \eqref{tau n,n}, recalling that $g_{n-2,n-2},g_{n-2,n}$ are known and that $\alpha\neq 0$, we determine $g_{n,n}$. From \eqref{tau n-1,n}, recalling that

\[g_{n-1,n-2},g_{n,n-2},g_{n,n}\]

are known and again that $\alpha\neq 0$, we determine $g_{n-1,n}$, hence the matrix $g$ is completely determined in this case too.

\end{proof}

\begin{definition}
Given distinct points $x,y,w,z\in\Sigma$, we define

\begin{equation}\label{def K}
K_{\sigma}(x,y,w,z)=N_{\sigma}(x,y)-N_{\sigma}(x,w)-N_{\sigma}(z,y)+N_{\sigma}(z,w).
\end{equation}

\end{definition}

Note that, fixing $w,z\in\Sigma$, $K_{\sigma}$, as a function of $x$, $y$, has the same asymptotic behaviour of $N_{\sigma}(x,y)$ as $x\rightarrow y$.

\begin{remark}
It is well-known that the knowledge of the full N-D map is equivalent to the knowledge of the boundary values of the Neumann kernel. It can also be verified that the local knowledge of the kernel implies knowing the local N-D map. Here we make precise the adjustments needed in the local determination of the kernel from the knowledge of the local map. The following lemma states that from $\mathcal{N}^{\Sigma}_{\sigma}$ one can determine locally $N_{\sigma}(x,y)$ up to a bounded function which is the sum of two terms $N_{\sigma}(x,w)$, $N_{\sigma}(z,y)-N_{\sigma}(z,w)$, one depending on $x$ only and the other depending on $y$ only.
\end{remark}

\begin{lemma}\label{lemma neumann map and K}
$\mathcal{N}^{\Sigma}_{\sigma}$ is known if and only if $K_{\sigma}$ is known for any $x,y,w,z\in\Sigma$.
\end{lemma}

\begin{proof}
For any $\varphi,\psi\in{C}^{0,1}_{0}(\Sigma)\cap_{0}H^{-\frac{1}{2}}$ we have

\begin{eqnarray}
\left<\psi, \mathcal{N}^{\Sigma}_{\sigma}\varphi\right> &=& \int_{\Sigma}\psi(\xi)dS(\xi)\int_{\Sigma} N_{\sigma}(\xi,\eta)\varphi(\eta)dS(\eta)\label{relation neumann map and K 1}\\
&=& \int_{\Sigma\times\Sigma}N_{\sigma}(\xi,\eta)\psi(\xi)\varphi(\eta)dS(\xi)\:dS(\eta)\label{relation neumann map and K 2}.
\end{eqnarray}

Note that the right hand side of

\begin{equation}\label{N-K}
N_{\sigma}(\xi,\eta)-K_{\sigma}(\xi,\eta,w,z)=N_{\sigma}(\xi,w)+N_{\sigma}(z,\eta)-N_{\sigma}(z,w)
\end{equation}

is a sum of terms which depend on at most one of the two variables $\xi$ and $\eta$. Recalling that $\varphi,\psi$ have zero average it follows that $N_{\sigma}(\xi,\eta)-K_{\sigma}(\xi,\eta,w,z)$ is orthogonal to $\psi(\xi)\varphi(\eta)$ in $L^2(\Sigma\times\Sigma)$, therefore \eqref{relation neumann map and K 2} leads to

\begin{equation}\label{relation neumann map and K final}
\left<\psi, \mathcal{N}^{\Sigma}_{\sigma}\varphi\right> = \int_{\Sigma\times\Sigma}K_{\sigma}(\xi,\eta,w,z)\psi(\xi)\varphi(\eta)\:dS(\xi)\:dS(\eta).
\end{equation}

Hence $K_{\sigma}$ uniquely determines $N^{\Sigma}_{\sigma}$. Vice versa, we pick

\begin{eqnarray*}
\psi(\xi) &=& \delta_{\varepsilon}(\xi;x)-\delta_{\varepsilon}(\xi;z),\\
\varphi(\eta) &=& \delta_{\varepsilon}(\eta;y)-\delta_{\varepsilon}(\eta;w),
\end{eqnarray*}

where $\delta_{\varepsilon}$ are approximate Dirac's delta functions on $\Sigma$ centered on the second argument. From \eqref{relation neumann map and K 2}, by letting $\varepsilon\rightarrow 0$ we can determine

\begin{equation*}
K_{\sigma}(x,y,w,z),
\end{equation*}

which concludes the proof.
\end{proof}


\section{Proof of the main result}

\begin{proof}[Proof of Theorem \ref{teorema principale}] Without loss of generality we can assume

\[\Sigma=\Sigma_1.\]

Let $\sigma^{(i)}$, for $i=1,2$ be two conductivities of type \eqref{conduttivita anisotrope} satisfying \eqref{unifellip}. If

\[\mathcal{N}^{\Sigma_1}_{\sigma^{(1)}}=\mathcal{N}^{\Sigma_1}_{\sigma^{(2)}},\]

then

\begin{equation}\label{global uniqueness2}
\sigma^{(1)}=\sigma^{(2)},\qquad\textnormal{in}\quad D_1.
\end{equation}

We shall proceed by induction. Let $D_K$ be a subdomain of $\Omega$, with $K\neq 1$ and recall that there exist $j_1,\dots , j_K\in\{1,\dots , N\}$ such that

\[D_{j_1}=D_1,\dots D_{j_K}=D_K,\]

with $D_{j_1},\dots D_{j_K}$ satisfying assumption $4(d)$. For simplicity, we rearrange the indices of these subdomains so that the above mentioned chain is simply denoted by $D_1, \dots, D_K, K\le N$. We assume that

\begin{equation}\label{sigma on the first K domains}
\sigma^{(1)}=\sigma^{(2)},\qquad\textnormal{in}\quad D_i,\qquad\textnormal{for\:every}\:i,\quad 1\leq i\leq K
\end{equation}

and show that

\[\sigma^{(1)}=\sigma^{(2)},\qquad\textnormal{in}\quad D_{K+1}\:\textnormal{too}.\]

We shall set

\[D=\left(\bigcup_{i=1}^{K}\overline{D_i}\right)^{\circ};\qquad E=\Omega\setminus\overline{D}.\]

We shall denote by $\mathcal{N}^{\Sigma_{K+1}}_{\sigma^{(i)}}$ the local N-D map for the domain $E$ relative to the conductivity $\sigma^{(i)}$ and localized on $\Sigma_{K+1}$, for $i=1,2$.

\begin{claim}\label{claim maps}
If $\mathcal{N}^{\Sigma_1}_{\sigma^{(1)}}=\mathcal{N}^{\Sigma_1}_{\sigma^{(2)}}$ and $\sigma^{(1)}=\sigma^{(2)}$ in $D$ then $\mathcal{N}^{\Sigma_{K+1}}_{\sigma^{(1)}}=\mathcal{N}^{\Sigma_{K+1}}_{\sigma^{(2)}}$.
\end{claim}

\noindent\textit{Proof of claim \ref{claim maps}.} Here we shall adapt some arguments already used in \cite{A-K}. Recall that up to a rigid transformation of coordinates we can assume that

\[P_1=0\qquad ;\qquad (\mathbb{R}^n\setminus\Omega)\cap B_{r_0}=\{(x^{\prime},x_n)\in B_{r_{0}}\:|\:x_n <\varphi(x^{\prime})\},\]

where $\varphi$ is a Lipschitz function such that

\[\varphi(0)=0\qquad\textnormal{and}\qquad ||\varphi||_{C^{0,1}(B_{r_{0}}^{\prime})}\leq Lr_0.\]

Denoting by

\[D_0=\left\{x\in(\mathbb{R}^n\setminus\Omega)\cap B_{r_0}\:\bigg|\:|x_i|<\frac{2}{3}r_0,\:i=1,\dots , n-1,\:\left|x_n-\frac{r_0}{6}\right|<\frac{5}{6}r_0\right\},\]

it turns out that the augmented domain $\Omega_0=\Omega\cup D_0$ is of Lipschitz class with constants $\frac{r_0}{3}$ and $\widetilde{L}$, where $\widetilde{L}$ depends on $L$ only. For any number $r\in \left(0,\frac{2}{3}r_0\right)$ we also denote

\[(D_0)_r = \left\{x\in D_0\:|\:dist(x,\Omega)>r\right\}.\]

For $i=1,2$ we consider the operator $L_i=\textnormal{div}(\sigma^{(i)}\nabla\cdot)$ in $\Omega$ and extend $\sigma^{(i)}$ to  $\widetilde{\sigma}^{(i)}$ on $\Omega_0$, by setting $\widetilde{\sigma}^{(i)}|_{D_0}=I$, where $I$ denotes the $n\times n$ identity matrix. For $y\in\Omega_0$ we define the modified Neumann kernel $\widetilde{N}_{\sigma^{(i)}}$ as the solution to

\begin{displaymath}
\left\{ \begin{array}{lll}
L_{i}\widetilde{N}_{\widetilde\sigma^{(i)}}^{\Omega}(\cdot,y)=-\delta(x-y), & \textnormal{in}\quad\Omega_{0}\\
\widetilde\sigma^{(i)}\nabla \widetilde{N}_{\widetilde\sigma^{(i)}}^{\Omega}\cdot\nu=0, & \textnormal{on}\quad\partial\Omega_{0}\cap\partial\Omega\\
\widetilde\sigma^{(i)}\nabla \widetilde{N}_{\widetilde\sigma^{(i)}}^{\Omega}\cdot\nu=-\frac{1}{\vert \partial\Omega_{0}\setminus\bar\Omega\vert}, & \textnormal{on}\quad\partial\Omega_{0}\setminus\bar\Omega.
\end{array} \right.
\end{displaymath}

Here we convene to normalize $\widetilde{N}_{\widetilde\sigma^{(i)}}^{\Omega}$, by prescribing

\[\int_{\partial\Omega_{0}} \widetilde{N}_{\widetilde\sigma^{(i)}}^{\Omega}(\cdot,y)\:dS(\cdot)=0.\]

Again, with this choice we obtain

\begin{equation}\label{symmetry Ni}
\widetilde{N}_{\widetilde\sigma^{(i)}}^{\Omega}(x,y) = \widetilde{N}_{\widetilde\sigma^{(i)}}^{\Omega}(y,x),\qquad\textnormal{for\:all}\quad x,y\in\Omega_0,\quad x\neq y.
\end{equation}

From now on we will simplify our notation by denoting

\[\widetilde{N}_{\widetilde\sigma^{(i)}}^{\Omega} = \widetilde{N}^{(i)}.\]

Given $\psi\in{C}^{0,1}(\partial E)$, with $\textnormal{supp}\psi\subset\Sigma^{K+1}$ and $\int_{\partial E}\eta =0$, we let $u^{(i)}$ solve

\begin{displaymath}
\left\{ \begin{array}{ll}
L_{i}u^{(i)}=0, & \textnormal{in}\quad E\\
\sigma^{(i)}\nabla u\cdot\nu=\psi, & \textnormal{on}\quad\partial{E}.
\end{array} \right.
\end{displaymath}

We consider a bounded extension operator

\[T:H^{\frac{1}{2}}(\partial E\cap\Omega)\longrightarrow H^{1}(\Omega),\]

such that, given $f\in H^{\frac{1}{2}}(\partial E\cap\Omega)$, we have

\[Tf\big\vert_{\Sigma_1}=0.\]

We denote

\begin{displaymath}
\overline{u}^{(i)}=\left\{ \begin{array}{ll}
u^{(i)}, & \textnormal{in}\quad E\\
T\left(u^{(i)}\big\vert_{\partial{E}\cap\Omega}\right), & \textnormal{in}\quad{D}.
\end{array} \right.
\end{displaymath}

Clearly $\bar{u}^{(i)}\in H^{1}(\Omega)$. For $x\in E$ we have

\begin{eqnarray}\label{u_i}
u^{(i)}(x) &=& -\int_{\Omega} \overline{u}^{(i)}(y)\:\textnormal{div}_y\left(\sigma^{(i)}(y)\nabla_y \widetilde{N}^{(i)}(y,x)\right)\:dy\nonumber\\
&=&-\int_{\partial\Omega} \overline{u}^{(i)}(y)\:\sigma^{(i)}(y)\nabla_y \widetilde{N}^{(i)}(y,x)\cdot\nu\:dS(y)\nonumber\\
&+&\int_{\Omega} \sigma^{(i)}(y)\nabla_{y}\overline{u}^{(i)}(y)\cdot \nabla_y \widetilde{N}^{(i)}(y,x)\:dy\nonumber\\
&=&\int_{E} \sigma^{(i)}(y)\nabla_{y}\overline{u}^{(i)}(y)\cdot \nabla_y \widetilde{N}^{(i)}(y,x)\:dy\nonumber\\
&+&\int_{D} \sigma^{(i)}(y)\nabla_{y}\overline{u}^{(i)}(y)\cdot \nabla_y \widetilde{N}^{(i)}(y,x)\:dy\nonumber\\
&=&\int_{\Sigma_{K+1}} \psi \widetilde{N}^{(i)}(y,x)\:dS(y) + \int_{D} \sigma^{(i)}(y)\nabla_{y}\overline{u}^{(i)}(y)\cdot \nabla_y \widetilde{N}^{(i)}(y,x)\:dy.
\end{eqnarray}

By differentiating under the integrals and by using Fubini, we form

\begin{eqnarray*}\label{product gradients}
& & \nabla_x u^{(1)}(x) \cdot \nabla_x u^{(2)}(x)\nonumber\\
& & =\int_{\Sigma_{K+1}\times\Sigma_{K+1}} \psi(y)\psi(z)\nabla_x\widetilde{N}^{(1)}(y,x)\cdot\nabla_x\widetilde{N}^{(2)}(z,x)\:dy\:dz\nonumber\\
& & +\int_{\Sigma_{K+1}\times{D}} \psi(y)\sigma^{(2)}_{lk}(z)\partial_{z_{l}}\overline{u}^{(2)}(z)\partial_{z_k}\left(\nabla_x\widetilde{N}^{(1)}(y,x)\cdot\nabla_x\widetilde{N}^{(2)}(z,x)\right)\:dy\:dz\nonumber\\
& & +\int_{D\times\Sigma_{K+1}} \psi(z)\sigma^{(1)}_{lk}(y)\partial_{y_{l}}\overline{u}^{(1)}(z)\partial_{y_k}\left(\nabla_x\widetilde{N}^{(2)}(z,x)\cdot\nabla_x\widetilde{N}^{(1)}(y,x)\right)\:dy\:dz\nonumber\\
& & +\int_{D\times D}\sigma^{(2)}_{lk}(z)\partial_{z_{l}}\overline{u}^{(2)}(z) \sigma^{(1)}_{nm}(y)\partial_{y_{n}}\overline{u}^{(1)}(z)\partial_{z_k}\partial_{y_m}\left(\nabla_x\widetilde{N}^{(2)}(z,x)\cdot\nabla_x\widetilde{N}^{(1)}(y,x)\right)\:dy\:dz.
\end{eqnarray*}

We define for $y,z\in D\cup D_0$

\begin{equation}\label{S}
S(y,z)=\int_{E} \left(\sigma^{(1)}(x)-\sigma^{(2)}(x)\right)\nabla_x\widetilde{N}^{(1)}(y,x)\cdot\nabla_x\widetilde{N}^{(2)}(z,x)\:dx.
\end{equation}

For any $y,z\in \left(\overline{D}\cup \overline{D}_0\right)^{\circ}$ we verify that

\begin{eqnarray}\label{S eps solutions}
& & \textnormal{div}_y\left(\sigma^{(1)}(y)\nabla_y\:S(y,z)\right) =0,\nonumber\\
& & \textnormal{div}_z\left(\sigma^{(2)}(z)\nabla_z\:S(y,z)\right) =0,
\end{eqnarray}

moreover

\begin{equation}\label{S as integral on Omega}
S(y,z) = \int_{\Omega} \left(\sigma^{(1)}(x)-\sigma^{(2)}(x)\right)\nabla_x\widetilde{N}^{(1)}(y,x)\cdot\nabla_x\widetilde{N}^{(2)}(z,x)=0
\end{equation}

because $\sigma^{(1)}=\sigma^{(2)}$ on $D$ by assumption. For $y,z\in D_{(\frac{r_0}{3})}$, being these singular points outside $\Omega$, by the identity \eqref{Alessandrini identity local N-D} we obtain

\[S(y,z)=\left<\sigma^{(1)}\nabla\widetilde{N}^{(1)}(y,\cdot)\cdot\nu,\left(\mathcal{N}_{\sigma^{(2)}}^{\Sigma_1} - \mathcal{N}_{\sigma^{(1)}}^{\Sigma_1}\right)\sigma^{(2)}\nabla \widetilde{N}^{(2)}(y,\cdot)\cdot\nu\right> = 0.\]

We recall that by the $C^{1,\alpha}$ regularity of the interfaces $\Sigma_{j_{k}}$ within $D$, $S(y,z)$  satisfies the unique continuation property in each variable $y,z\in \left(\overline{D}\cup \overline{D}_0\right)^{\circ}$ hence

\begin{equation}\label{S eps zero}
S(y,z)=0,\qquad\textnormal{for\:any}\quad y,z\in D.
\end{equation}

Consequently we obtain

\begin{eqnarray}\label{product gradients 2}
& & \int_{E} \left(\sigma^{(1)}(x)-\sigma^{(2)}(x)\right)\nabla_x u^{(1)}(x)\cdot\nabla_x u^{(2)}(x)\:dx\nonumber\\
 & & =\int_{\Sigma_{K+1}\times\Sigma_{K+1}} \psi(y)\psi(z)S(y,z)\:dy\:dz\nonumber\\
& & +\int_{\Sigma_{K+1}\times{D}} \psi(y)\sigma^{(2)}_{lk}(z)\partial_{z_{l}}\overline{u}^{(2)}(z)\partial_{z_k}S(y,z)\:dy\:dz\nonumber\\
& & +\int_{D\times\Sigma_{K+1}} \psi(z)\sigma^{(1)}_{lk}(y)\partial_{y_{l}}\overline{u}^{(1)}(z)\partial_{y_k}S(y,z)\:dy\:dz\nonumber\\
& & +\int_{D\times D}\sigma^{(2)}_{lk}(z)\partial_{z_{l}}\overline{u}^{(2)}(z) \sigma^{(1)}_{nm}(y)\partial_{y_{n}}\overline{u}^{(1)}(z)\partial_{z_k}\partial_{y_m}S(y,z)\:dy\:dz=0.
\end{eqnarray}

Hence

\begin{equation}\label{identification Neumann map K+1}
\left<\psi,\left(\mathcal{N}^{\Sigma_{K+1}}_{\sigma^{(1)}}-\mathcal{N}^{\Sigma_{K+1}}_{\sigma^{(2)}}\right)\:\psi\right> =  \int_{E}\left(\sigma^{(2)}(x)-\sigma^{(1)}(x)\right)\nabla_x u^{(1)}(x)\cdot\nabla_x u^{(2)}(x)\:dx =0,
\end{equation}

which concludes the proof of the claim.$\qquad\qquad\Box$\\

From $\mathcal{N}^{\Sigma_{K+1}}_{\sigma^{(1)}}=\mathcal{N}^{\Sigma_{K+1}}_{\sigma^{(2)}}$ and by Lemma \ref{lemma sigma constant} we obtain

\[\sigma^{(1)}(x)=\sigma^{(2)}(x),\qquad\textnormal{for\:any}\quad x\in \Sigma_{K+1},\]

hence

\[\sigma^{(1)}(x)=\sigma^{(2)}(x),\qquad\textnormal{for\:any}\quad x\in D_{K+1},\]

which concludes the proof.

\end{proof}

\begin{ex}\label{example}
Let $v=(v',v_n)\in\mathbb{R}^{n}_{+}$ be an arbitrary point (note that $v_n>0$). Consider the matrix

\[
\sbox0{{{\huge\mbox{{$I_{(n-1)}$}}}}}
M=\left(
\begin{array}{c|c}
\usebox{0}&\makebox[\wd0]{\Large $v'$}\\
\hline
  \vphantom{\usebox{0}}\makebox[\wd0]{\Large $0'^T$}&\makebox[\wd0]{\Large $v_n$}
\end{array}
\right),
\]

where we understand

\[v'= \left(
\begin{array}{c}
v_1\\
\vdots\\
v_{n-1}
\end{array}
\right)
\]

and $0'$ denotes the column null $(n-1)$-vector. $M$ is a linear transformation of $\mathbb{R}^{n}_{+}$ into itself which fixes the boundary $\Pi_n$. Following the calculations in the proof of Lemma \ref{lemma N and tau}, let us form

\[\sigma=\frac{QQ^{T}}{\det{Q}},\]

where $Q=M^{-1}$. $\sigma$ is the push-forward of the isotropic homogeneous conductivity $I$ through the change of coordinates $x=M\xi$. In this case

\begin{eqnarray*}
g=M^{T}M &=&  \sbox0{{{\huge\mbox{{$I_{(n-1)}$}}}}}
\left(
\begin{array}{c|c}
\usebox{0}&\makebox[\wd0]{\Large $0'$}\\
\hline
  \vphantom{\usebox{0}}\makebox[\wd0]{\Large $v'^{T}$}&\makebox[\wd0]{\Large $v_n$}
\end{array}
\right)\sbox0{{{\huge\mbox{{$I_{(n-1)}$}}}}}
\left(
\begin{array}{c|c}
\usebox{0}&\makebox[\wd0]{\Large $v'$}\\
\hline
  \vphantom{\usebox{0}}\makebox[\wd0]{\Large $0'^{T}$}&\makebox[\wd0]{\Large $v_n$}
\end{array}
\right)\\
&=&\sbox0{{{\huge\mbox{{$I_{(n-1)}$}}}}}
\left(
\begin{array}{c|c}
\usebox{0}&\makebox[\wd0]{\Large $v'$}\\
\hline
  \vphantom{\usebox{0}}\makebox[\wd0]{\Large $v'^{T}$}&\makebox[\wd0]{\Large $|v'|^{2}+v^{2}_n$}
\end{array}
\right).
\end{eqnarray*}

Hence

\[g_{(n-1)}=I_{(n-1)},\]

for any choice of $v\in\mathbb{R}^{n}_{+}$.\\

In other words, the whole family of anisotropic conductivities

\begin{equation*}
\sigma =\frac{QQ^{T}}{\det{Q}}=
%
%
%
v_n\sbox0{{{\huge\mbox{{$I_{(n-1)}$}}}}+\Large $\frac{1}{v^{2}_n}v'v'^{T}$}
\left(
\begin{array}{c|c}
\usebox{0}&\makebox[\wd0]{\Large $-\frac{1}{v^{2}_n}v'$}\\
\hline
  \vphantom{\usebox{0}}\makebox[\wd0]{\Large $-\frac{1}{v^{2}_n}v'^{T}$}&\makebox[\wd0]{\Large $\frac{1}{v^{2}_n}$}
\end{array}
\right)
\end{equation*}

is such that

\[N_{\sigma}(x',y')=N_{I}(x',y')\qquad\textnormal{for\:all}\quad x',y'\in\Pi_n.\]

That is, any such $\sigma$ is indistinguishable from the identity $I$ when the corresponding N-D map (or D-N map) on $\Pi_n$ is given.
\end{ex}

\section*{Acknowledgments}
Giovanni Alessandrini acknowledges support by FRA2014
`Problemi inversi per PDE, unicit\`a, stabilit\`a, algoritmi',
Universit\`a degli Studi di Trieste.

Maarten V. de Hoop acknowledges and sincerely thanks the Simons Foundation for financial support.

Romina Gaburro acknowledges the support of the Department of Mathematics and Statistics, University of Limerick, during a visit to Rice University, where
part of the research for the preparation of this paper was carried out in the Fall 2015/16.


\end{document}